\documentclass{amsart}

\usepackage[colorlinks=true, urlcolor=blue, citecolor=blue, linkcolor=blue, hypertexnames=false]{hyperref}
\usepackage{amssymb}
\usepackage{amsfonts}
\usepackage{amsmath}
\usepackage{amsthm}
\usepackage{mathrsfs}
\usepackage{dsfont}
\usepackage{mathtools}
\usepackage{graphicx}
\usepackage{tikz}

\usepackage{aliascnt}

\newtheorem{thm}{Theorem}[section]

\newaliascnt{lem}{thm}

\aliascntresetthe{lem}

\newaliascnt{prp}{thm}  
\newtheorem{prp}[prp]{Proposition}
\aliascntresetthe{prp}

\newaliascnt{cor}{thm}  

\aliascntresetthe{cor}

\theoremstyle{definition}

\newaliascnt{dfn}{thm}  
\newtheorem{dfn}[dfn]{Definition}
\aliascntresetthe{dfn}

\numberwithin{equation}{section}

\author{Tristan Bice}
\address{
Federal University of Santa Catarina\\
Florianopolis\\
Brazil
}
\email{Tristan.Bice@gmail.com}
\thanks{This research has been supported by a CAPES (Brazil) postdoctoral fellowship through the program ''Science without borders'', PVE project 085/2012.}

\begin{document}

\title{Type Decomposition in Posets}

\begin{abstract}
Motivated by the classical type decomposition of von Neumann algebras, and various more recent extensions to other structures, we develop a type decomposition theory for general posets.
\end{abstract}

\maketitle

\section{Introduction}

The basic idea of type decomposition is simple \textendash\, to take a structure $S$ and decompose it into two parts $T$ and $U$, so $T$ will possess certain properties while $U$ will possess diametrically opposite properties, i.e. $T$ and $U$ will be of different `types'.  Various theorems might be dependent on these properties, and so type decomposition allows us to find the largest part of $S$ on which these theorems hold.  More interestingly, it sometimes occurs that a theorem can be proved for different types, but with different arguments, and proving the theorem for $S$ requires first applying these different arguments to $T$ and $U$.

These ideas originated in Murray and von Neumann's fundamental work in \cite{MurrayvonNemann1936}, where the structures in question were what we now call von Neumann algebras, and the structural properties they based their decompositions on involved commutativity and finiteness.  For example, they showed that any von Neumann algebra $A$ contains subalgebras $B$ and $C$ with $A=B\oplus C$ where $B$ is semifinite, meaning every non-zero projection dominates a non-zero finite projection, and $C$ is purely infinite, meaning it contains no non-zero finite projections whatsoever.  Following this, type decompositions were obtained for various other more general structures, usually lattices or orthoposets with some remnant of distributivity, like modularity or orthomodularity, together with some additional structure, like an equivalence relation or partial binary function satisfying certain properties (see \cite{Loomis1955}, \cite{Kaplansky1955}, \cite{MaedaF1959}, \cite{GoodearlWehrung2005} and \cite{FoulisPulmannova2013}).

The purpose of the present paper is to unify and generalize these type decomposition results, which we do in \autoref{IcapZ}, \autoref{cZI} and \autoref{homthm}.  This has its own intrinsic value, but actually our primary motivation is that we want to apply the theory to more general structures, annihilators in C*-algebras in particular (which generalize projections in von Neumann algebras).  This requires eliminating, or at least weakening, some of the assumptions made on structures in previous type decomposition results.  For example, annihilators in a C*-algebra may not be orthomodular, as is often the case with ortholattices derived from polarities, and the equivalence relation on annihilators in a C*-algebra that naturally generalizes Murray-von Neumann equivalence may not be orthogonally divisible (see \cite{Bice2014}).

One option would be take some collection of structural properties that holds for annihilators in C*-algebras and use them to prove the relevant type decomposition theorems in this more general context.  Indeed, whenever type decomposition is to be applied to some new kind of structure, this is generally the approach that is taken, and why a number of similar results have been proved in slightly different contexts.  Thus we feel it is time to take a more minimalist approach, which is what we aim for in this paper, where we use the weakest assumptions possible and introduce them only when necessary.  While this leads to some results which may, at first sight, seem somewhat technical, the advantage is that it more clearly illustrates precisely which assumptions are being used and where.  This, in turn, should facilitate the discovery of any future applications and generalizations.

\section{Type Decomposition}\label{TD}

%We will frequently deal with supremums $\vee$ and infimums $\wedge$ which, of course, do not always exist in an arbitrary poset.  Consequently, we adopt the convention that equality denotes Kleene equality, i.e. an expression of the form $X=Y$, where $X$ and $Y$ are terms involving supremums and infimums, means that $X$ and $Y$ are either both undefined or both defined and equal.  In paricular, as a term consisting of a single variable $x$ is always well defined, an expression like $w=x\vee(y\wedge z)$ means that $x\vee(y\wedge z)$ is well-defined and equal to $w$, i.e. $y$ and $z$ have an infimum $y\wedge z$ which, in turn, has a supremum with $x$ which equals $w$.  Likewise, any statement of the form $X\in y$, means that $X$ is a well-defined member of $y$.  Indeed, while type decomposition of a poset $\mathbb{P}$ is always done with respect to some kind of central subset $Z$ which, for the optimal results, must be a pseudocomplemented lattice, there is no need to assume that the entirety of $\mathbb{P}$ is a lattice or possesses any kind of complementation.

First, let us review some basic order theoretic terminology.

\begin{dfn}
If $\mathbb{P}$ is a poset then $I\subseteq\mathbb{P}$ is
\begin{enumerate}
\item a \emph{lower set} if $p\leq q\in I\Rightarrow p\in I$, for all $p\in\mathbb{P}$.
\item \emph{completely upwards directed} if $S\subseteq I\Rightarrow\exists q\in I\forall p\in S(p\leq q)$.
\item a \emph{complete ideal} if $I$ is a completely upwards directed lower set.
\item\label{sublat} an \emph{upper (lower) complete sublattice} if $\bigvee S\in I$ ($\bigwedge S\in I$) for all $S\subseteq I$.
\end{enumerate}
\end{dfn}
As we are dealing with posets rather than complete lattices, infimums $\bigwedge$ and supremums $\bigvee$ do not always exist, and their existence in \eqref{sublat} is implicitly part of the definition.  In other words, when we say that $I$ is an upper complete sublattice we mean that every subset $S$ of $I$ actually has a supremum in $\mathbb{P}$, which also lies in $I$.  So any upper complete sublattice of $\mathbb{P}$ will be an upper complete lattice in its own right, even if $\mathbb{P}$ itself is not a lattice.  It follows that $I$ will also be a lower complete lattice, as infimums are supremums of lower bounds, although it is important to note that these infimums in $I$ may not agree with those in $\mathbb{P}$, i.e. $I$ may not be a (lower) sublattice of $\mathbb{P}$.

Also, $I$ will be a complete ideal if and only if $I=[0,q]=\{p\in\mathbb{P}:p\leq q\}$, where $q=\bigvee I$.  And $I\subseteq\mathbb{P}$ will be simultaneously a complete ideal and upper complete sublattice if and only if, for all $S\subseteq I$, \[S\subseteq I\quad\Leftrightarrow\quad\bigvee S\in I.\]  For type decomposition, we relativize this with respect to some $Z\subseteq\mathbb{P}$, which will play the role traditionally occupied by the centre.  And from now on we assume our posets always have a lower bound, which we denote by $0$.

%$0=\bigwedge\mathbb{P}$ and $Z\subseteq\mathbb{P}$
\begin{dfn}\label{compdef}
For a poset $\mathbb{P}$ and $Z\subseteq\mathbb{P}$, $S\subseteq\mathbb{P}$ is \emph{$Z$-disjoint} if we have $f:S\rightarrow Z$ with $p\leq f(p)$ and $f(p)\wedge f(q)=0$, for distinct $p,q\in S$, and $I\subseteq\mathbb{P}$ is \emph{$Z$-complete} if
\begin{enumerate}
\item\label{compdef1} $\bigvee S\in I$, whenever $S\subseteq I$ is $Z$-disjoint, and
\item\label{compdef2} $p,q\in I$, whenever $p\vee q\in I$ and $\{p,q\}$ is $Z$-disjoint.
\end{enumerate}
\end{dfn}

In particular, if $I$ is $Z$-complete then $0=\bigvee\emptyset\in I$.  Also note that if $I$ is a lower set then \eqref{compdef2} is immediate.  In fact, if $\mathbb{P}$ is a lattice then the $Z$-complete lower sets are precisely the $\mathscr{P}_Z(\mathbb{P})$-ideals defined in \cite{Niederle2006}, where $\mathscr{P}_Z(\mathbb{P})$ denotes the collection of all $Z$-disjoint subsets of $\mathbb{P}$.  While if $\mathbb{P}$ is an effect algebra with centre $Z$ then the $Z$-complete lower sets are precisely the strongly type determining (STD) subsets defined in \cite{FoulisPulmannova2010} \S4.  And in this case $\mathbb{P}$ itself will be $Z$-complete precisely when $\mathbb{P}$ is centrally orthocomplete, as defined in \cite{FoulisPulmannova2010b}.

The other basic ingredient for type decomposition is centrality.

\begin{dfn}
If $\mathbb{P}$ is a poset and $S\subseteq\mathbb{P}$ then $Z\subseteq\mathbb{P}$ is \emph{$S$-central} if, for all $p\in S$ and $y\in Z$, we have $z\in Z$ with $y\wedge z=0$ and $p=q\vee r$, for some $q\leq y$ and $r\leq z$.
%\[p=q\vee r,\quad\textrm{for some }q\leq y\textrm{ and }r\leq z.\]
\end{dfn}

If $\mathbb{P}$ is an ortholattice then $z\in\mathbb{P}$ is central in the usual sense (see \cite{MacLaren1964} \S3 or \cite{Kalmbach1983} \S3 Theorem 1) if and only if $\{z,z^\perp\}$ is $\mathbb{P}$-central.  More generally, in an arbitrary lattice (see \cite{MaedaMaeda1970} Definition (4.12)), or even an arbitrary poset (see \cite{GudderHaskins1974}) we call $z\in\mathbb{P}$ central iff there exists $z^\perp\in\mathbb{P}$ such that $\mathbb{P}$ is canonically isomorphic to $[0,z]\times[0,z^\perp]$, i.e. isomorphic via the maps $\psi:\mathbb{P}\rightarrow[0,z]\times[0,z^\perp]$ and $\phi:[0,z]\times[0,z^\perp]\rightarrow\mathbb{P}$ defined by
\begin{equation}\label{centralmaps}
\psi(p)=(p\wedge z,p\wedge z^\perp)\quad\textrm{and}\quad\phi(p,q)=p\vee q.
\end{equation}
Thus the centre $\mathrm{C}(\mathbb{P})$, i.e. the subset of all central elements of $\mathbb{P}$, is $\mathbb{P}$-central by the above definition.  In applications $Z$ will often be the centre but $\mathbb{P}$-central subsets can, in general, be very much larger than the centre (basically because of the differing order of quantifiers).  For example, $\mathbb{P}$ itself will be $\mathbb{P}$-central if $\mathbb{P}$ is a meet-semilattice that is section complemented (meaning sections, i.e. intervals with lower bound $0$, are complemented posets), while $\mathrm{C}(\mathbb{P})$ is necessarily a Boolean (i.e. distributive complemented) lattice.  There is also value in dealing with a strict subset $Z$ of $\mathrm{C}(\mathbb{P})$, as the lattice $\mathrm{C}(\mathbb{P})$ itself may not be complete.

We are now ready for our first general type decomposition.

\begin{thm}\label{IcapZ}
If $\mathbb{P}$ is a poset, $I,Z\subseteq\mathbb{P}$ are $Z$-complete and $Z$ is $Z$-central then $I\cap Z$ is a complete ideal of $Z$ and upper complete sublattice of $\mathbb{P}$.  Moreover, $z=\bigvee I\cap Z$ is the unique $z\in Z$ such that, for all $y\in Z$,
\begin{equation}\label{IcapZeq}
y\wedge z=0\quad\Leftrightarrow\quad[0,y]\cap I\cap Z=\{0\}.
\end{equation}
\end{thm}

\begin{proof}
If $z\in I\cap Z$ and $y\in[0,z]\cap Z$ then, as $Z$ is $Z$-central, we have $x\in Z$ with $x\wedge y=0$ and $z=p\vee q$, for some $p\leq x$ and $q\leq y$.  As $y\leq z$, we have $z=p\vee y$ which, as $I$ is $Z$-complete, means $y\in I$.  As $y\in[0,z]\cap Z$ was arbitrary, $I\cap Z$ is a lower set in $Z$.

Now say we have $(z_\alpha)\subseteq I\cap Z$.  Let $y_0=z_0$.  As $Z$ is $Z$-central, we have $y\in Z$ with $y\wedge z_{0}=0$ and $z_1=p\vee y_1$, for some $p\leq z_0$ and $y_1\leq y$.  As $I$ and $Z$ are $Z$-complete, we have $y_1\in I\cap Z$ and $y_0\vee y_1\in I\cap Z$.  But $y_0\vee y_1=z_0\vee p\vee y_1=z_0\vee z_1$ so $z_0\vee z_1\in I\cap Z$.  By recursion we can continue in this way to obtain a transfinite sequence $(y_\alpha)\subseteq I\cap Z$ with $y_\alpha\wedge y_\beta=0$, whenever $\alpha\neq\beta$, and $\bigvee_{\alpha<\beta}y_\alpha=\bigvee_{\alpha<\beta}z_\alpha$, for all $\beta$, which, as $I$ and $Z$ are $Z$-complete, means that $\bigvee z_\alpha\in I\cap Z$.  As $(z_\alpha)\subseteq I\cap Z$ was arbitrary, $I\cap Z$ is closed under arbitrary supremums.

In particular, we have $z=\bigvee I\cap Z\in I\cap Z$, and clearly $[0,y]\cap I\cap Z=\{0\}$ for any $y\in Z$ with $y\wedge z=0$.  While if $y\in Z$ then, as $Z$ is $Z$-central, we have $x\in Z$ with $x\wedge y=0$ and $z=p\vee q$, for some $p\leq x$ and $q\leq y$.  As $I$ and $Z$ are $Z$-complete, $q\in I\cap Z$.  So if $[0,y]\cap I\cap Z=\{0\}$ then $q=0$ and hence $z=p\leq x$, which means $y\wedge z=0$.

Now say we have another $z'\in Z$ satisfying \eqref{IcapZeq}.  As $Z$ is $Z$-central, we have $y\in Z$ with $y\wedge z=0$ and $z'=p\vee q$ with $p\leq y$ and $q\leq z$.  Thus $[0,y]\cap I\cap Z=\{0\}$ and hence $[0,p]\cap I\cap Z=\{0\}$.  As $Z$ is $Z$-complete, $p\in Z$ which, by our assumption on $z'$, means $0=z'\wedge p=p$ and hence $z'=q\leq z$.  The same argument with $z$ and $z'$ reversed shows that $z\leq z'$.
\end{proof}

If the $z$ above has a complement $y$, then \eqref{IcapZeq} shows that $y$ has the opposite type to $z$, i.e. we get a complementary type decomposition.  Also, \eqref{IcapZeq} shows that the elements of $Z$ of this opposite type will also form a complete ideal in $Z$ precisely when $z$ has a \emph{pseudocomplement} $z^\perp=\bigvee\{y\in Z:y\wedge z=0\}$ (see \cite{Birkhoff1967} Ch 5 \S8).  And this will yield a direct product type decomposition precisely when $z$ is central.
%In particular, this happens if $Z$ is a Boolean lattice, like the centre of $\mathbb{P}$.  It turns out that if the $Z$ above is pseudocomplemented, i.e. every element of $Z$ has a pseudocomplement, then $Z$ in fact has to be a Boolean lattice (see \autoref{Booleanposet}).

If $\mathbb{P}$ is a section complemented complete lattice then $\mathbb{P}$ itself will be $\mathbb{P}$-central and $\mathbb{P}$-complete, in which case \autoref{IcapZ} with $Z=\mathbb{P}$ says that the $\mathbb{P}$-complete subsets are precisely the complete ideals of $\mathbb{P}$.  A slightly more interesting situation arises, as mentioned in the introduction, when $\mathbb{P}$ is the complete lattice of projections $\mathcal{P}(A)$ of a von Neumann algebra $A$ and $Z$ is the complete sublattice of central projections $\mathcal{P}(A\cap A')$.  Then, letting $I$ be the set of finite projections, i.e. those $p\in\mathcal{P}(A)$ such that $pAp$ is a finite von Neumann algebra, we see that $I$ is $Z$-complete, because a direct sum of finite von Neumann algebras is again finite.  Thus \autoref{IcapZ} applies to give a unique central finite projection $z$ such that $z^\perp$ is properly infinite, i.e. does not dominate any central finite projection.  In terms of classical type decompostion terminology for von Neumann algebras, $zA(=zAz)$ consists of the type $\mathrm{I}_n$ parts, for finite $n$, together with the type $\mathrm{II}_1$ part of $A$, while $z^\perp A$ consists of the type $\mathrm{I}_\infty$, type $\mathrm{II}_\infty$ and type $\mathrm{III}$ part of $A$.  Likewise, we can apply \autoref{IcapZ} with the abelian projections as $I$ instead, and then $zA$ would be the type $\mathrm{I}_1$ part of $A$.  Indeed, the classical type decomposition of von Neumann algebras is obtained from combining these decompositions with other type decompositions obtained from using $I$ and $Z$ in different ways.  We shall examine one of these next, and complete the picture with the type decompositions obtained in \S\ref{HD}.

\begin{dfn}
%If $R$ is a binary relation on a set $S$ and $T\subseteq S$, we define \[T^R=\{s\in S:\forall t\in T(tRs)\}.\]
If $\mathbb{P}$ is a poset and $Z\subseteq\mathbb{P}$ then $\mathrm{c}_Z(p)$ is the \emph{$Z$-cover} of $p$ when \[\mathrm{c}_Z(p)=\bigwedge([p,1]\cap Z).\]
\end{dfn}

To ensure that $\mathrm{c}_Z(p)$ is defined and in $Z$, for all $p\in\mathbb{P}$, we assume that $Z\subseteq\mathbb{P}$ is a lower complete sublattice of $\mathbb{P}$ (in particular, $1=\bigwedge\emptyset\in Z$).  In this case, note that $S\subseteq\mathbb{P}$ will be $Z$-disjoint iff $\mathrm{c}_Z(s)\wedge\mathrm{c}_Z(t)=0$, for all distinct $s,t\in S$.  Also, we denote supremums in $Z$ by $\vee_Z$, as they may differ from supremums $\vee$ in $\mathbb{P}$, and we denote the set of all $Z$-covers of elements of $I$ by $\mathrm{c}_ZI$.

We now have the machinery for our second general type decomposition.

\begin{thm}\label{cZI}
If $\mathbb{P}$ is a poset, $Z$ is a lower complete sublattice of $\mathbb{P}$, $Z$ is $\mathbb{P}$-central and $I\subseteq\mathbb{P}$ is $Z$-complete then $\mathrm{c}_ZI$ is an upper complete sublattice of $Z$.  Moreover, $z=\bigvee_Z\mathrm{c}_ZI$ is the unique $z\in Z$ such that, for all $y\in Z$,
\begin{equation}\label{cZIeq}
y\wedge z=0\quad\Leftrightarrow\quad[0,y]\cap I=\{0\}.
\end{equation}
\end{thm}

\begin{proof} Say we have $(p_\alpha)\subseteq I$.  Let $q_0=p_0$.  As $Z$ is $\mathbb{P}$-central, we have $z\in Z$ with $z\wedge\mathrm{c}_Z(p_0)=0$ and $p_1=r\vee q_1$, for some $r\leq\mathrm{c}_Z(p_0)$ and $q_1\leq z$.  As $I$ is $Z$-complete, $q_1\in I$ and $q_0\vee q_1\in I$.  Moreover, \[\mathrm{c}_Z(q_0\vee q_1)=\mathrm{c}_Z(p_0)\vee_Z\mathrm{c}_Z(r)\vee_Z\mathrm{c}_Z(q_1)=\mathrm{c}_Z(p_0)\vee_Z\mathrm{c}_Z(p_1).\]  Recursively continuing in this way we obtain a transfinite sequence $(q_\alpha)\subseteq I$ with $\mathrm{c}_Z(q_\alpha)\wedge\mathrm{c}_Z(q_\beta)=0$, whenever $\alpha\neq\beta$, and $\mathrm{c}_Z(\bigvee_{\alpha<\beta}q_\alpha)=\bigvee_{Z,\alpha<\beta}\mathrm{c}_Z(p_\alpha)$, for all $\beta$, which, as $I$ is $Z$-complete, means that $\bigvee_Z\mathrm{c}_Z(p_\alpha)=\mathrm{c}_Z(\bigvee p_\alpha)\in\mathrm{c}_ZI$.  As $(p_\alpha)\subseteq I$ was arbitrary, $\mathrm{c}_ZI$ is closed under arbitrary supremums in $Z$.

In particular, $z=\bigvee_Z\mathrm{c}_ZI=\mathrm{c}_Z(p)$, for some $p\in I$, and clearly $[0,y]\cap I=\{0\}$ for any $y\in Z$ with $y\wedge z=0$.  While if $y\in Z$ then, as $Z$ is $\mathbb{P}$-central, we have $x\in Z$ with $x\wedge y=0$ and $p=q\vee r$, for some $q\leq x$ and $r\leq y$.  As $I$ is $Z$-complete, $r\in I$, so if $[0,y]\cap I=\{0\}$ then $r=0$.  This means $p=q$ so $z=\mathrm{c_Z}(p)\leq x$ and hence $y\wedge z=0$.

Now say we have another $z'\in Z$ satisfying \eqref{cZI}.  As $Z$ is $\mathbb{P}$-central, we have $y\in Z$ with $y\wedge z=0$ and $z'=p\vee q$ with $p\leq y$ and $q\leq z$.  Thus $[0,y]\cap I=\{0\}$ and hence $[0,\mathrm{c}_Z(p)]\cap I=\{0\}$.  By our assumption on $z'$, $0=z'\wedge\mathrm{c}_Z(p)=\mathrm{c}_Z(p)$ and hence $z'=q\leq z$.  The same argument with $z$ and $z'$ reversed shows that $z\leq z'$.
\end{proof}

%Note any pseudocomplemented $\mathbb{P}$-central $Z$ must actually be an orthoposet contained in the pseudocentre of $\mathbb{P}$, i.e. $\mathbb{P}$-central $Z$ will be pseudocomplemented iff $Z^\perp=Z\subseteq\mathrm{Z}_\mathbb{P}$.

Again considering the case $\mathbb{P}=\mathcal{P}(A)$ and $Z=\mathcal{P}(A\cap A')$, where $A$ is von Neumann algebra, \autoref{cZI}  applies when $I$ is the set of finite projections, showing that $A$ contains a a unique central semifinite projection $z$ such that $z^\perp$ is purely infinite.  In terms of classical type decompostion terminology for von Neumann algebras, $zA$ consists of the type $\mathrm{I}$ and $\mathrm{II}$ parts while $z^\perp A$ is the type $\mathrm{III}$ part of $A$.  Again as before, we can apply \autoref{cZI} with the abelian projections as $I$ instead, and then $z$ would be the unique central discrete projection such that $z^\perp$ is continuous  (see \cite{Berberian1972} \S15 Definition 3 for this terminology), and $zA$ would be the type $\mathrm{I}$ part of $A$, while $z^\perp A$ would consist of the type $\mathrm{II}$ and $\mathrm{III}$ parts of $A$.

\section{Modularity}

As already mentioned, previous type decomposition results have focused on the case that $Z$ is contained in the centre of $\mathbb{P}$.  And some previous proofs of these results have indeed used the fact that the centre is distributive.  The previous section shows that distributivity is not vital for type decomposition, but nonetheless there are some extra things we can say in this case, or even when $Z$ is assumed to satisfy the following weaker assumption.

\begin{dfn}\label{Pmod}
For a poset $\mathbb{P}$, $Z\subseteq\mathbb{P}$ is \emph{$\mathbb{P}$-modular} if, whenever $y,z\in Z$, $y\wedge z=0$, $p\leq y$, $q\leq z$ and $p\vee q$ exists, we have $q=z\wedge(p\vee q)$.
\end{dfn}
So \autoref{Pmod} is saying that $Z$ is $\mathbb{P}$-modular if all disjoint pairs in $Z$ are modular pairs.\footnote{in a weak sense \textendash\, for lattices there is a standard notion of modular pair (see \cite{MaedaMaeda1970}) but for arbitrary posets there are a number of other possible generalizations (see \cite{ThakarePawarWaphare2004}) based on the canonical embedding of $\mathbb{P}$ in its Dedekind-MacNeille completion.}  The first thing this allows us to obtain is the following slightly different characterization of $Z$-completeness.  This shows that the $Z$-complete subsets are precisely the $P$-properties in \cite{MaedaF1959} Definition 1.3 (when $\mathbb{P}$ is a complete lattice with centre $Z$) and also the type-determining (TD) sets defined in \cite{FoulisPulmannova2010} \S4 (when $\mathbb{P}$ is a centrally orthocomplete effect algebra (COEA) with centre $Z$).

\begin{prp}\label{pwedgez}
If $\mathbb{P}$ is a poset and $Z\subseteq\mathbb{P}$ is $\mathbb{P}$-modular then $I\subseteq\mathbb{P}$ is $Z$-complete if (and only if, when $Z$ is $\mathbb{P}$-central)
\begin{itemize}
\item[\eqref{compdef1}] $\bigvee S\in I$, whenever $S\subseteq I$ is $Z$-disjoint, and
\item[\eqref{compdef2}$'$] $p\wedge z\in I$, whenever $p\in I$ and $z\in Z$.
\end{itemize}
\end{prp}

\begin{proof}
If $I$ satisfies the above conditions then, for any $p,q\in\mathbb{P}$ with $\{p,q\}$ $Z$-disjoint and $p\vee q\in I$, we have $y,z\in Z$ with $p\leq y$, $q\leq z$ and $y\wedge z=0$ which, by $\mathbb{P}$-modularity, yields $p=(p\vee q)\wedge y\in I$ and $q=(p\vee q)\wedge z\in I$, so $I$ is $Z$-complete.  On the other hand, if $Z$ is $\mathbb{P}$-central then, for any $p\in I$ and $z\in Z$ we have $y\in Z$ with $y\wedge z=0$ and $p=q\vee r$, for some $q\leq y$ and $r\leq z$.  Thus if $I$ is $Z$-complete then $p\wedge z=(q\vee r)\wedge z=r\in I$, again using the $\mathbb{P}$-modularity of $Z$.
\end{proof}

%As mentioned after \autoref{IcapZ}, the type decomposition obtained there becomes symmetric, in the sense that those $z$ of the opposite type also form a complete ideal in $Z$, precisely when $z$ has a pseudocomplement $z^\perp=\bigvee\{y\in Z:y\wedge z=0\}$.  Here we investigate what else can be said about pseudocomplemented subsets.  Firstly, we get the following symmetric characterization of $Z$-completeness.

While on the topic of $Z$-completeness, let us point out that pseudocomplements also allow for the following more symmetric characterization.

\begin{prp}
If $\mathbb{P}$ is a $Z$-complete poset, for pseudocomplemented $Z\subseteq\mathbb{P}$, then $I\subseteq\mathbb{P}$ will be $Z$-complete if and only if, for all $Z$-disjoint $S\subseteq I$, \[S\subseteq I\quad\Leftrightarrow\quad\bigvee S\in I.\]
\end{prp}

\begin{proof}
The `if' part is immediate.  Conversely, assume $I$ is $Z$-complete and say we have $Z$-disjoint $S\subseteq I$ with $\bigvee S\in I$ and let $f:S\rightarrow Z$ be as in \autoref{compdef}.  For $s\in S$, and $t\in T=S\setminus\{s\}$ we have $f(s)\wedge f(t)=0$ and hence $f(t)\leq f(s)^\perp$.  As $T$ is $Z$-disjoint and $\mathbb{P}$ is $Z$-complete, $\bigvee T$ exists and $\bigvee T\leq f(s)^\perp$ so $s\wedge\bigvee T=0$.  As $s\vee\bigvee T=\bigvee S\in I$ and $I$ is $Z$-complete, $s\in I$.  As $s\in S$ was arbitrary, we are done.
\end{proof}

Using $Z$-modularity rather than $\mathbb{P}$-modularity (which implies $Z$-modularity as long as $Z$ is a (upper) sublattice of $\mathbb{P}$), we also obtain a result on $Z$-covers.

\begin{prp}
If $Z$ is a $Z$-modular $\mathbb{P}$-central lower complete sublattice of $\mathbb{P}$ then, for all $p\in\mathbb{P}$ and $z\in Z$, we have $q\leq p,z$ with $\mathrm{c}_Z(q)=\mathrm{c}_Z(p)\wedge z$.
\end{prp}

\begin{proof}
As $Z$ is $\mathbb{P}$-central, given $p\in\mathbb{P}$ and $z\in Z$ we have $y\in Z$ with $y\wedge z=0$ and $p=q\vee r$, for some $q\leq z$ and $r\leq y$.  Take $x\in Z$ with $q\leq x$.  Then $p=q\vee r\leq(x\wedge\mathrm{c}_Z(p)\wedge z)\vee_Z(\mathrm{c}_Z(p)\wedge y)\in Z$, so \[\mathrm{c}_Z(p)\leq(x\wedge\mathrm{c}_Z(p)\wedge z)\vee_Z(\mathrm{c}_Z(p)\wedge y)\leq\mathrm{c}_Z(p).\]  As $Z$ is $Z$-modular, $z\wedge\mathrm{c}_Z(p)=x\wedge\mathrm{c}_Z(p)\wedge z\leq x$.  As $x\geq q$ was an arbitrary element of $Z$, and $q\leq\mathrm{c}_Z(p)\wedge z\in Z$, we have $\mathrm{c}_Z(q)=\mathrm{c}_Z(p)\wedge z$.
\end{proof}

If, in the situation above, $p\wedge z$ exists then the $q$ above satisfies $q\leq p\wedge z$ so
\begin{equation}\label{cpz}
\mathrm{c}_Z(q)\leq\mathrm{c}_Z(p\wedge z)\leq\mathrm{c}_Z(p)\wedge z=\mathrm{c}_Z(q).
\end{equation}
In the particular case that $\mathbb{P}$ is a centrally orthocomplete effect algebra (COEA) and $Z$ is its centre, this shows that $\mathrm{c}_Z$ is a hull mapping, according to \cite{FoulisPulmannova2010b} Definition 5.1.
%If $Z$ is also $\mathbb{P}$-modular then, for all $p\in\mathbb{P}$ and $z\in Z$, $p\wedge z$ will indeed exist so \[\mathrm{c}_Z(p\wedge z)=\mathrm{c}_Z(p)\wedge z.\]
%So if $Z$ is a $Z$-modular $\mathbb{P}$-central lower complete sublattice of $\mathbb{P}$, $p\in\mathbb{P}$, $z\in Z$ and 

We can also use $Z$-modularity to show that the upper complete sublattice obtained in \autoref{cZI} is additionally a lower set in $Z$. %comes from $Z$-modularity.

\begin{thm}\label{cZIMZ}
If $\mathbb{P}$ is a poset, $Z$ is a $Z$-modular lower complete sublattice of $\mathbb{P}$, $Z$ is $\mathbb{P}$-central and $I\subseteq\mathbb{P}$ is $Z$-complete then $\mathrm{c}_ZI$ is complete ideal of $Z$.
\end{thm}

\begin{proof}
If $p\in I$ and $y\in[0,\mathrm{c}_Z(p)]\cap Z$ then, as $Z$ is $\mathbb{P}$-central, we have $x\in Z$ with $x\wedge y=0$ and $p=q\vee r$, for some $q\leq x$ and $r\leq y$.  Thus $\mathrm{c}_Z(p)=\mathrm{c}_Z(q)\vee_Z\mathrm{c}_Z(r)$ and hence $y=y\wedge_Z\mathrm{c}_Z(p)=y\wedge_Z(\mathrm{c}_Z(q)\vee_Z\mathrm{c}_Z(r))=\mathrm{c}_Z(r)$, as $Z$ is $Z$-modular.  As $I$ is $Z$-complete, $r\in I$ so $y\in\mathrm{c}_ZI$.  As $y\in[0,\mathrm{c}_Z(p)]\cap Z$ was arbitrary, $\mathrm{c}_ZI$ is a lower set in $Z$.  Also, $Z$ is an upper complete sublattice of $Z$, by \autoref{cZI}.
\end{proof}

%Type decompositions are always done with resepect to some subset $Z$ of $\mathbb{P}$, which usually has a much nicer structure than that of $\mathbb{P}$ itself.  Indeed, for the best results, $Z$ will need to be a Boolean lattice but there is no need to assume the entirety of $\mathbb{P}$ is a lattice or possesses any kind of complementation.  In fact, in previous papers on type decomposition, $Z$ is usually assumed to be a very particular Boolean lattice, namely the centre of $\mathbb{P}$ (see below for the definition).

%However, there are a couple of problems with this.  The first is that this is often overkill, as only certain properties of the centre are necessary to prove the theorems in question, properties which could also hold for more general subsets, and this goes against our stated aim of using the minimal assumptions necessary.  Perhaps more importantly, even if $\mathbb{P}$ is complete lattice, its centre may not be complete, a crucial property necessary for type decomposition.  This is no doubt at least partly why previous type decomposition results have used extra assumptions on $\mathbb{P}$, like orthomodularity, to ensure that the centre is indeed complete.  However, sometimes it turns out that we can identify a particular subset of the centre that is indeed complete, even if the whole centre is not.  In order that type decomposition can be done with these and more general subsets, we shall reverse the usual course and focus on properties of $Z$ first, and then use these to define and investigate various notions of centre.

\section{Complete Relations}

Type decomposition is always done with respect to some $Z$-complete $I\subseteq\mathbb{P}$.  The next natural question to ask is where these $Z$-complete subsets might come from.  It turns out that there are two major sources of $Z$-complete subsets, relations and classes.\footnote{Somewhat surprisingly, it even turns out that often the same $Z$-complete subset can, with the appropriate relation and class, be derived either way.  For example, the Boolean class corresponds to the central equivalence relation $\sim_Z$, as discussed below, the modular class corresponds to the perspectivity relation (having a common complement), while the orthomodular class corresponds to the orthoperspectivity relation (having a common orthogonal complement).}  However, using classes requires restricting $Z$ to subsets of the centre and, moreover, making additional assumptions on $\mathbb{P}$, like assuming $\mathbb{P}$ is section semicomplemented (see \cite{MaedaMaeda1970} Theorem (5.13)).  In the present paper we wish to avoid such assumptions, so we shall focus solely on relations.

%and leave the reader to consult the previous literature for the class approach (e.g. for the class approach in effect algebras see \cite{FoulisPulmannova2010} \S4).

Note below $\mathbb{P}\times\mathbb{P}$ has the product order, i.e. $(p,q)\leq(r,s)\Leftrightarrow p\leq r$ and $q\leq s$.

\begin{dfn}
For a poset $\mathbb{P}$, $Z\subseteq\mathbb{P}$ and binary relation $\precsim$ on $\mathbb{P}$, we say $\precsim$ is \emph{$Z$-complete} if $\precsim$ is $=_Z$-complete, considering $\precsim$ and $=_Z$ as subsets of $\mathbb{P}\times\mathbb{P}$.
\end{dfn}

As $=_Z$, i.e. $\{(z,z):z\in Z\}$, is $\mathbb{P}\times\mathbb{P}$-modular precisely when $Z$ is $\mathbb{P}$-modular, we have the following rephrasing of \autoref{pwedgez}.  Note, however, that even if $Z$ is $\mathbb{P}$-central, $=_Z$ may not be $\mathbb{P}\times\mathbb{P}$-central, unless $Z$ is also pseudocomplemented.  Another important thing to note is that $Z$-complete relations need only be centrally divisible, by \eqref{compdef2}$'$ below, rather than orthogonally divisible, as required for the Loomis dimension relations in \cite{Loomis1955} page 2 (B) and the Sherstnev-Kalinin congruences in \cite{FoulisPulmannova2013} Definition 4.1 (SK3d).  This is important because, as mentioned in the introduction, the analog of Murray-von Neumann equivalence for annihilators in a C*-algebras is centrally, but possibly not orthogonally, divisible.

\begin{prp}\label{pwedgezrel}
For a poset $\mathbb{P}$ and $\mathbb{P}$-modular $Z\subseteq\mathbb{P}$, a binary relation $\precsim$ on $\mathbb{P}$ will be $Z$-complete if (and only if, when $=_Z$ is $\mathbb{P}\times\mathbb{P}$-central), for $Z$-disjoint $(z_\alpha)$,
\begin{itemize}
\item[\eqref{compdef1}] $\bigvee p_\alpha\precsim\bigvee q_\alpha$, whenever $p_\alpha\precsim q_\alpha$ and $p_\alpha,q_\alpha\leq z_\alpha$, for all $\alpha$, and
\item[\eqref{compdef2}$'$] $p\wedge z\precsim q\wedge z$, whenever $p\precsim q$ and $z\in Z$.
\end{itemize}
\end{prp}

Just like with projections in von Neumann algebras, $\precsim$-finite elements can be defined for any relation $\precsim$ on a poset $\mathbb{P}$.

\begin{dfn}\label{findef}
For a poset $\mathbb{P}$ and binary relation $\precsim$ on $\mathbb{P}$, $p\in\mathbb{P}$ is \emph{$\precsim$-finite} if \[p\precsim q\leq p\quad\Rightarrow\quad p=q,\] for all $q\in\mathbb{P}$.  We denote the set of all $\precsim$-finite elements of $\mathbb{P}$ by $\mathrm{F}_\precsim$.
\end{dfn}

And again, just like with projections, if $\precsim$ is a $Z$-complete relation then the $\precsim$-finite elements will form a $Z$-complete subset and so the type decomposition results in \autoref{IcapZ} and \autoref{cZI} can be applied.

%We will primarily be interested in the $\precsim$-$=$-finite elements, which we refer to simply as $\precsim$-finite and denote by $\mathrm{F}_\precsim$.

%$Z^\perp=Z\subseteq\mathrm{S}_\mathbb{P}$
%\begin{prp}\label{fincom}
%If $\mathbb{P}$ is a $Z$-complete poset, $Z$ is pseudocomplemented, $\mathbb{P}$-central and $\mathbb{P}$-modular, and $\precsim\ %\subseteq\mathbb{P}\times\mathbb{P}$ is reflexive and $Z$-complete then $\mathrm{F}_\precsim$ is $Z$-complete.
%\end{prp}

\begin{prp}\label{fincom}
If $\mathbb{P}$ is a $Z$-complete poset, $Z=Z^\perp$ is contained in the centre of $\mathbb{P}$ and $\precsim$ is a $Z$-complete reflexive binary relation on $\mathbb{P}$ then $\mathrm{F}_\precsim$ is $Z$-complete.
\end{prp}

\begin{proof}
Assume $S\subseteq\mathrm{F}_\precsim$ is $Z$-disjoint, as witnessed by $f:S\rightarrow Z$, but $p=\bigvee S\notin\mathrm{F}_\precsim$, so $p\precsim q<p$, for some $q\in\mathbb{P}$.  Thus $s\nleq q$, for some $s\in S$, and then $s=p\wedge f(s)\precsim q\wedge f(s)<s$, contradicting $s\in S\subseteq F_\precsim$.

Now assume $p\in\mathrm{F}_\precsim$ and $z\in Z$ but $p\wedge z\notin\mathrm{F}_\precsim$, so we have $q<p\wedge z$ with $p\wedge z\precsim q$.  As $\precsim$ is reflexive and $\precsim$ is $Z$-complete, $p=(p\wedge z)\vee(p\wedge z^\perp)\precsim q\vee(p\wedge z^\perp)<p$, contradicting $p\in\mathrm{F}_\precsim$.
\end{proof}

The next natural question to ask is where these $Z$-complete relations might come from.  If $\mathbb{P}$ is defined from some algebraic structure then $Z$-complete relations can often be derived from this.  Murray-von Neumann equivalence of projections and its natural generalization to annihilators in a C*-algebra are examples.  The order structure of $\mathbb{P}$ can also be used to define $Z$-complete relations, like perspectivity (having a common complement), but as with the class approach, this often requires $\mathbb{P}$ to satisfy some additional assumptions.  We can also use $Z$ itself to define important $Z$-complete relations, as we now show.

\begin{dfn}
For a poset $\mathbb{P}$ and $Z\subseteq\mathbb{P}$, we define relations $\sim_Z$ and $\precsim_Z$ on $\mathbb{P}$ by
\[p\precsim_Zq\ \Leftrightarrow\ [p,1]\cap Z\subseteq[q,1]\cap Z\quad\textrm{and}\quad p\sim_Zq\ \Leftrightarrow\ [p,1]\cap Z=[q,1]\cap Z.\]
\end{dfn}

Also, we call $\mathbb{P}$ \emph{$Z$-directed} if $p\vee q$ exists, for all $Z$-disjoint $\{p,q\}\subseteq\mathbb{P}$.

\begin{thm}
If $\mathbb{P}$ is a $Z$-complete poset and $Z$ is a $Z$-directed $Z$-modular lower sublattice of $\mathbb{P}$ then $\precsim_Z$ and $\sim_Z$ are $Z$-complete.
\end{thm}

\begin{proof}
If $p_\lambda\precsim_Zq_\lambda$, for all $\lambda\in\Lambda$, and $(p_\lambda,q_\lambda)_{\lambda\in\Lambda}$ is $=_Z$-disjoint then $(p_\lambda)$ and $(q_\lambda)$ are $Z$-disjoint which, as $\mathbb{P}$ is $Z$-complete, means $p=\bigvee p_\lambda$ and $q=\bigvee q_\lambda$ exist.  If $z\in Z\cap[q,1]$ then $q_\lambda\leq z$ and hence, as $p_\lambda\precsim_Zq_\lambda$, $p_\lambda\leq z$, for all $\lambda\in\Lambda$, so $p\leq z$.  As $z\in Z$ was arbitrary, $p\precsim_Zq$.

On the other hand, if $p\vee q\precsim p'\vee q'$ and $\{(p,p'),(q,q')\}$ is $=_Z$-disjoint then we have $y,z\in Z$ with $p,p'\leq y$, $q,q'\leq z$ and $y\wedge z=0$.  For any $x\in Z$ with $p'\leq x$, $x\wedge y\wedge z=0$ and $p'\vee q'\leq(x\wedge y)\vee_Zz\in Z$ and hence $p\leq p\vee q\leq(x\wedge y)\vee_Zz$ so $p\leq y\wedge((x\wedge y)\vee_Zz)=x\wedge y\leq x$.  As $x$ was arbitrary, $p\precsim_Zp'$ and, likewise, $q\precsim_Zq'$.  Thus $\precsim_Z$ and $\sim_Z=(\precsim_Z\cap\succsim_Z)$ are $Z$-complete.
\end{proof}

In fact $\precsim_Z$ is the weakest $Z$-complete relation with $\{0\}^\succsim=\{p\in\mathbb{P}:p\precsim0\}=\{0\}$.

\begin{prp}
If $Z$ is $\mathbb{P}$-central and $\precsim$ is a $Z$-complete relation on $\mathbb{P}$ with $\{0\}^\succsim=\{p\in\mathbb{P}:p\precsim0\}=\{0\}$ then $p\precsim q\Rightarrow p\precsim_Zq$, for all $p,q\in\mathbb{P}$.
\end{prp}

\begin{proof}
If $p\precsim q$ and $z\in Z\cap[q,1]$, we have $y\in Z$ with $y\wedge z=0$ and $p=r\vee s$, for some $r\leq y$ and $s\leq z$.  Then $r,0\leq y$ and $s,q\leq z$ and $r\vee s=p\precsim q=0\vee q$ so $r\precsim0$ which, by assumption, means $r=0$ and hence $p=s\leq z$.  As $z\in Z\cap[q,1]$ was arbitrary, $p\precsim_Zq$.
\end{proof}

%we have $(z_\lambda)\subseteq Z$ with $z_\alpha\wedge z_\beta=0$, for $\alpha\neq\beta$, and $p_\lambda,q_\lambda\leq z_\lambda$, for all $\lambda\in\Lambda$.  As $\mathbb{P}$ is $Z$-complete, $\bigvee p_\lambda$ and $\bigvee q_\lambda$ exist, and if $z\in Z$ satisfies 

If $Z$ is a lower complete sublattice of $\mathbb{P}$ then we immediately see that
\[p\precsim_Zq\ \Leftrightarrow\ \mathrm{c}_Z(p)\leq\mathrm{c}_Z(p)\quad\textrm{and}\quad p\sim_Zq\ \Leftrightarrow\ \mathrm{c}_Z(p)=\mathrm{c}_Z(p).\]
Moreover, $\mathrm{c}_Z$ actually characterizes $\mathrm{F}_{\precsim_Z}=\mathrm{F}_{\sim_Z}$ in the following way.

%need $Z$-completeness or at least $Z$-directedness to ensure that $q\vee s$ exists
%need $\mathbb{P}$-moduarity
\begin{thm}
If $\mathbb{P}$ is $Z$-directed and $Z$ is a $\mathbb{P}$-central $\mathbb{P}$-modular $Z$-modular lower complete sublattice of $\mathbb{P}$ then $p\in\mathrm{F}_{\precsim_Z}$ if and only if $[0,p]=\{p\wedge z:z\in Z\}$.  Moreover, in this case $[0,p]$ is isomorphic to $[0,\mathrm{c}_Z(p)]\cap Z$ via the maps \[q\mapsto\mathrm{c}_Z(q)\quad\textrm{and}\quad z\mapsto p\wedge z.\]
\end{thm}

\begin{proof} If $p\notin\mathrm{F}_{\precsim_Z}$ then we have $q\in\mathbb{P}$ with $p\precsim_Zq<p$.  Then $\mathrm{c}_Z(q)=\mathrm{c}_Z(p)$ so $p\wedge z=p$, for any $z\geq q$, and hence $q\notin\{p\wedge z:z\in Z\}$.

Conversely, say $p\in\mathrm{F}_{\precsim_Z}$ and $q\leq p$.  As $Z$ is $\mathbb{P}$-central, we have $z\in Z$ with $\mathrm{c}_Z(q)\wedge z=0$ and $p=r\vee s$, for some $r\leq\mathrm{c}_Z(q)$ and $s\leq z$.  As $\mathbb{P}$ is $Z$-directed, $q\vee s$ exists and
\[\mathrm{c}_Z(p)=\mathrm{c}_Z(r\vee s)=\mathrm{c}_Z(r)\vee_Z\mathrm{c}_Z(s)\leq\mathrm{c}_Z(q)\vee_Z\mathrm{c}_Z(s)=\mathrm{c}_Z(q\vee s),\]
and hence $p\precsim_Z q\vee s\leq p$ which, as $p\in\mathrm{F}_{\precsim_Z}$, means $p=q\vee s$.  As $Z$ is $\mathbb{P}$-modular, $q=p\wedge\mathrm{c}_Z(q)$.  On the other hand, for $z\in[0,\mathrm{c}_Z(p)]\cap Z$, we have $\mathrm{c}_Z(p\wedge z)=\mathrm{c}_Z(p)\wedge z=z$, by \eqref{cpz}, so the given maps are indeed isomorphisms.
\end{proof}

The $\precsim_Z$-finite elements go by several other names in more restrictive contexts, as can be seen from the above characterization.  In effect algebras they are called monads (see \cite{FoulisPulmannova2010b} Definition 5.12), in complete lattices (with $Z=\mathrm{C}(\mathbb{P})$) they are called lowest elements (see \cite{MaedaF1959} Definition 2.2 and Theorem 2.1) and in orthomodular lattices they are called simple elements (see \cite{Loomis1955} page 2).  When $Z=\mathrm{C}(\mathbb{P})$ and $\mathbb{P}$ has the relative centre property (see \cite{Chevalier1991}), i.e. when we have $\mathrm{C}([0,p])=\{p\wedge z:z\in\mathrm{C}(\mathbb{P})\}$ for all $p\in\mathbb{P}$, the above result shows that $p$ is $\precsim_Z$-finite precisely when $[0,p]$ is a Boolean lattice.  Such elements are called Boolean in \cite{FoulisPulmannova2010b} Definition 5.12 and $D$-elements in \cite{Kaplansky1955} \S9.  And if we go back to our favourite example where $\mathbb{P}=\mathcal{P}(A)$ and $Z=\mathcal{P}(A\cap A')$, for some von Neumann algebra $A$, then we see that the $\precsim_Z$-finite elements are precisely the abelian projections.  Thus the $\precsim_Z$-finite elements give us an analog of abelian projections, with which we can even do type decomposition, in a very general class of posets $\mathbb{P}$ with a distinguished subset $Z$.

%and this is an example of the phenomenon mentioned at the beginning of this section, where the same $Z$-complete subset can be obtained either from a relation, namely $\precsim_Z$, or a class, namely the Boolean lattices

%If the $Z$ above is a Boolean lattice then so is $[0,p]$, for any $p\in\mathrm{F}_{\precsim_Z}$.  And if $Z$ is the centre of $\mathbb{P}$ then every $p\in\mathrm{F}_{\precsim_Z}$ is \emph{subcentral} in the terminology of \cite{FoulisPulmannova2010b} Definition 5.16, i.e. $\mathrm{C}_{[0,p]}=\{p\wedge z:z\in\mathrm{C}_\mathbb{P}\}$.

%If $\mathbb{P}$ is $Z$-directed and $Z$ is a $\mathbb{P}$-central $\mathbb{P}$-modular $Z$-modular lower complete sublattice of $\mathbb{P}$ then $p\in\mathrm{F}_{\precsim_Z}$ iff $[0,p]$ is isomorphic to $[0,\mathrm{c}_Z(p)]\cap Z$ via \[q\mapsto\mathrm{c}_Z(q)\quad\textrm{and}\quad z\mapsto p\wedge z.\]

%\[\mathrm{c}_Z(q\vee s)=\mathrm{c}_Z(q)\vee_Z\mathrm{c}_Z(s)\geq\mathrm{c}_Z(r)\vee_Z\mathrm{c}_Z(s)=\mathrm{c}_Z(r\vee s)=\mathrm{c}_Z(p)\]

\section{Homogeneous Decompositions}\label{HD}

One conspicuous absence in previous order theoretic treatments of type decomposition is an analog of the type $\mathrm{I}_n$ parts in the classical von Neumann algebra type decomposition.  These can be obtained immediately if one has a dimension function on $\mathbb{P}$, but the construction of a dimension function requires a significant amount of extra structure (see \cite{Maeda1955}, \cite{Kalinin1976} and \cite{GoodearlWehrung2005}).  Here we present an elementary method for obtaining such decompositions, at least in the orthocomplemented case, more in the spirit of \cite{Berberian1972} \S18.

\begin{dfn}\label{homdef}
Given a binary relation $R$ on a poset $\mathbb{P}$ we call $p\in\mathbb{P}$ \emph{$R$-homogeneous} if $p=\bigvee S$ for some $S\subseteq\mathbb{P}$ with $S\times S\subseteq R\ \cup=$, i.e. $sRt$ for all distinct $s,t\in S$.  We say such a $p$ has \emph{order} $\kappa$ when $|S|=\kappa$ and denote the set of order $\kappa$ $R$-homogeneous $p$ by $I_\kappa$.  Given $Z\subseteq\mathbb{P}$, we call $p\in\mathbb{P}$ \emph{$R$-subhomogeneous} if $p=\bigvee S$ for some $Z$-disjoint $S\subseteq\mathbb{P}$ consisting of $R$-homogeneous elements.
\end{dfn}

Note that order is not, in general, uniquely defined, and $I_\kappa$ consists of all those $p$ for which there is at least one set $S$ of cardinality $\kappa$ witnessing its $R$-homogeneity.

Given an orthoposet $\mathbb{P}$ (see \cite{Kalmbach1983} \S2 page 16) and $I,Z\subseteq\mathbb{P}$, we define the canonical homogeneity relation $\mathrm{H}=\mathrm{H}_{I,Z}=I\times I\ \cap\perp\cap\sim_Z$, i.e. \[p\mathrm{H}q\quad\Leftrightarrow\quad p,q\in I,p\perp q\textrm{ and }p\sim_Zq.\]  Note that below we follow standard order terminology and call $I\subseteq\mathbb{P}$ \emph{order-dense} when every $p\in\mathbb{P}$ dominates some non-zero $q\in I$.  We also call an orthoposet $\mathbb{P}$ \emph{orthocomplete} when $\bigvee S$ exists, for any pairwise orthogonal $S\subseteq\mathbb{P}$.\footnote{We can turn any orthoposet $\mathbb{P}$, in which orthogonal pairs have a supremum, into a pre-effect algebra (see \cite{ChajdaKuhr2012}) by defining $p+q=p\vee q$, whenever $p\perp q$.  \autoref{homthm} could also be generalized to orthocomplete pre-effect algebras, although we omit the details, as we do not wish to introduce a fundamental new structure at this stage.}

\begin{thm}\label{homthm}
For an orthocomplete orthoposet $\mathbb{P}$, complete sublattice $Z^\perp=Z\subseteq\mathrm{C}(\mathbb{P})$ and order-dense $Z$-complete $I\subseteq\mathbb{P}$, $1=\bigvee\mathbb{P}$ is $\mathrm{H}_{I,Z}$-subhomogeneous.  If, further, $I_\lambda\cap I_\kappa\cap Z=\{0\}$, whenever $\lambda\neq\kappa$, then there are unique orthogonal $(z_\kappa)\subseteq Z$ with $z_\kappa\in I_\kappa$, for all $\kappa$, and $1=\bigvee z_\kappa$.
\end{thm}

\begin{proof}
By \autoref{cZI}, we may recursively define $p_\alpha\in I_\alpha$ so that $\mathrm{c}_Z(p_\alpha)=\bigvee\mathrm{c}_ZI_\alpha$, where $I_\alpha=I\cap[0,(\bigvee_{\beta<\alpha}p_\beta)^\perp]$.  Let $z_\alpha=\mathrm{c}_Z(p_\alpha)^\perp\wedge\bigwedge_{\beta<\alpha}\mathrm{c}_Z(p_\beta)$.  If we had $z=z_\alpha\wedge(\bigvee_{\beta<\alpha}p_\beta)^\perp\neq0$ then the order density of $I$ would give us non-zero $p\in I\cap[0,z]\subseteq I_\alpha$ and hence $p\leq\mathrm{c}_Z(p_\alpha)$ even though $p\leq z\leq z_\alpha\leq\mathrm{c}_Z(p_\alpha)^\perp$, a contradiction.  Thus, as $z_\alpha\in\mathrm{C}(\mathbb{P})$, $z_\alpha=z_\alpha\wedge\bigvee_{\beta<\alpha}p_\beta=\bigvee_{\beta<\alpha}(z_\alpha\wedge p_\beta)$.  By \eqref{cpz}, \[\mathrm{c}_Z(z_\alpha\wedge p_\beta)=z_\alpha\wedge\mathrm{c}_Z(p_\beta)=z_\alpha,\] for each $\beta<\alpha$, so each $z_\alpha$ is $\mathrm{H}_{I,Z}$-homogeneous.  As the $(p_\alpha)$ are orthogonal, they must be eventually $0$.  Thus $(\bigvee z_\alpha)^\perp=\bigwedge\mathrm{c}_Z(p_\alpha)=0$ so the $(z_\alpha)$ witness the $\mathrm{H}_{I,Z}$-subhomogeneity of $1$.

As $I$ is $Z$-complete so is $I_\kappa$, for each $\kappa$, so we can join together resulting homogeneous elements of $Z$ of the same order to obtain $(z_\kappa)\subseteq Z$ with $z_\kappa\in I_\kappa$, for all $\kappa$.  Uniqueness now follows from \autoref{IcapZ}, as $I_\lambda\cap I_\kappa\cap Z=\{0\}$, for $\lambda\neq\kappa$, means that $z_\kappa=\bigvee I_\kappa\cap Z$, for all $\kappa$.
\end{proof}

When $\mathbb{P}=\mathcal{P}(A)$ and $Z=\mathcal{P}(A\cap A')$, for some type $\mathrm{I}$ von Neumann algebra $A$, and $I$ is the set of abelian projections, the above theorem does indeed apply and then $z_\kappa A$ is none other than the type $\mathrm{I}_\kappa$ part of $A$.  However, actually verifying that the required hypthoses are satisfied here is not as easy as it was for the type decomposition results in \S\ref{TD}, and so we now explain this in a little more detail.

Firstly, for order density, let $p\in\mathcal{P}(A)$ be an abelian projection with $\mathrm{c}(p)=1$.  For any non-zero $q\in\mathcal{P}(A)$, we then have $q\mathrm{c}(p)=q\neq0$ so $pAq\neq\{0\}$.  Thus we can find a non-zero partial isometry $u\in pAq$.  As $uu^*\leq p$ and $p$ is abelian, so is $uu^*$ and hence $u^*u\leq q$ is abelian too.  As $q\in\mathcal{P}(A)$ was arbitrary, we are done.  Essentially the same argument applies more generally to abelian annihilators in a discrete C*-algebra, except that it takes more effort to show that the equivalence relation generalizing Murray-von Neumann equivalence also always preserves abelian-ness (see \cite{Bice2014} Corollary 3.54).

On the other hand, to show that $I_\lambda\cap I_\kappa\cap Z=\{0\}$, i.e. that the order of a homogeneous central projection is uniquely defined, is far less trivial.  If $\lambda$ or $\kappa$ is finite then it follows from the fact that the finite projections form an (order) ideal in $\mathcal{P}(A)$ (see \cite{Berberian1972} \S17 Theorem 2).  We do not know if this holds more generally for finite annihilators in C*-algebras, but we can still use representations to show that a finite supremum of \emph{abelian} annihilators is finite (see \cite{Bice2014} \S3.9) which is sufficient, at least for (ortho)separable C*-algebras.  While for infinite $\lambda$ or $\kappa$ the proof uses the local orthoseparability of von Neumann algebras (see \cite{Berberian1972} \S18 Exercise 10) and does not even hold in general for projections in type I AW*-algebras (see \cite{Ozawa1985}), let alone annihilators in more general discrete C*-algebras.

Also, we should point out that in the definition of a homogeneous projection in a von Neumann algebra it is usually Murray von Neumann equivalence $\sim_\mathrm{MvN}$ that is used, as in \cite{Berberian1972} \S18 Definition 1, rather than central equivalence $\sim_Z$ as done here.  However, this makes no difference, as one can use generalized comparability (\cite{Berberian1972} \S14 Definition 1) to show that $\sim_\mathrm{MvN}$ coincides with $\sim_Z$ on abelian projections. 

\newpage

\bibliography{maths}{}
\bibliographystyle{alphaurl}

\end{document}